\documentclass[12pt]{amsart}

\usepackage{amsmath,amssymb,amsthm,url,bm,microtype,parskip,enumitem,mathtools,tikz-cd}
\usepackage{xcolor}
\usepackage[colorlinks]{hyperref}
\usepackage{mathrsfs}  

\newtheorem{lemma}{Lemma}[section]
\newtheorem{prop}[lemma]{Proposition}
\newtheorem{cor}[lemma]{Corollary}
\newtheorem{thm}[lemma]{Theorem}
\newtheorem{example}[lemma]{Example}
\newtheorem{thm?}[lemma]{Theorem?}

\newtheorem{ques}{Question}

\newtheorem{defn}{Definition}
\newtheorem{remark}[lemma]{Remark}

\newtheorem{notation}{Notation}

\begin{document}
\title{Generalized period-index problem with an application to quadratic forms}

\author{Saurabh Gosavi}

\email{gosavis@biu.ac.il}

\address{Department of Mathematics, Faculty of Exact Sciences, Bar-Ilan University \\ Ramat-Gan \\ Israel \\ 5290001. \\}

\newcommand{\etalchar}[1]{$^{#1}$}
\newcommand{\F}{\mathbb{F}}
\newcommand{\et}{\textrm{\'et}}
\newcommand{\ra}{\ensuremath{\rightarrow}}
\newcommand{\FF}{\F}
\newcommand{\Z}{\mathbb{Z}}
\newcommand{\N}{\mathcal{N}}
\newcommand{\ch}{char}
\newcommand{\R}{\mathbb{R}}
\newcommand{\PP}{\mathbb{P}}
\newcommand{\pp}{\mathfrak{m}}

\newcommand{\Q}{\mathbb{Q}}
\newcommand{\tpqr}{\widetilde{\triangle(p,q,r)}}
\newcommand{\ab}{\operatorname{ab}}
\newcommand{\Aut}{\operatorname{Aut}}
\newcommand{\gk}{\mathfrak{g}_K}
\newcommand{\gq}{\mathfrak{g}_{\Q}}
\newcommand{\OQ}{\overline{\Q}}
\newcommand{\Out}{\operatorname{Out}}
\newcommand{\End}{\operatorname{End}}
\newcommand{\Gon}{\operatorname{Gon}}
\newcommand{\Gal}{\operatorname{Gal}}
\newcommand{\CT}{(\mathcal{C},\mathcal{T})}
\newcommand{\ttop}{\operatorname{top}}
\newcommand{\lcm}{\operatorname{lcm}}
\newcommand{\Div}{\operatorname{Div}}
\newcommand{\OO}{\mathcal{O}}
\newcommand{\rank}{\operatorname{rank}}
\newcommand{\tors}{\operatorname{tors}}
\newcommand{\IM}{\operatorname{IM}}
\newcommand{\CM}{\operatorname{CM}}
\newcommand{\Frac}{\operatorname{Frac}}
\newcommand{\Pic}{\operatorname{Pic}}
\newcommand{\coker}{\operatorname{coker}}
\newcommand{\Cl}{\operatorname{Cl}}
\newcommand{\loc}{\operatorname{loc}}
\newcommand{\GL}{\operatorname{GL}}
\newcommand{\PSL}{\operatorname{PSL}}
\newcommand{\Frob}{\operatorname{Frob}}
\newcommand{\Hom}{\operatorname{Hom}}
\newcommand{\Coker}{\operatorname{\coker}}
\newcommand{\Ker}{\ker}
\renewcommand{\gg}{\mathfrak{g}}
\newcommand{\sep}{\operatorname{sep}}
\newcommand{\new}{\operatorname{new}}
\newcommand{\Ok}{\mathcal{O}_K}
\newcommand{\ord}{\operatorname{ord}}
\newcommand{\MM}{\mathcal{M}}
\newcommand{\mm}{\mathfrak{m}}
\newcommand{\Ohell}{\OO_{p^{\infty}}}
\newcommand{\ff}{\mathfrak{f}}
\renewcommand{\N}{\mathbb{N}}
\newcommand{\Gm}{\mathbb{G}_m}
\newcommand{\Spec}{\operatorname{Spec}}
\newcommand{\MaxSpec}{\operatorname{MaxSpec}}
\newcommand{\qq}{\mathfrak{q}}
\newcommand{\cof}{\operatorname{cf}}
\newcommand{\lub}{\operatorname{lub}}
\newcommand{\glb}{\operatorname{glb}}
\newcommand{\redbox}{{\color{red}{\blacksquare}}}
\newcommand{\CKn}{\text{CK}(n)}
\newcommand{\CKnm}{\text{CK}(n,m)}
\newcommand{\CK}{{\rm CK}}
\newcommand{\sing}{\operatorname{Sing}}
\newcommand{\br}{\operatorname{Br}}
\newcommand{\bd}{\operatorname{GBrd}}
\newcommand{\sbd}{\operatorname{GBrd}}
\newcommand{\ram}{\operatorname{Ram}}
\newcommand{\comm}[1]{}
\newcommand{\ind}{ind}
\newcommand{\mind}{mind}
\newcommand{\per}{per}
\newcommand{\degr}{deg}
\newcommand{\md}{mod}
\newcommand{\cd}{cd}
\newcommand{\Divisor}{\operatorname{div}}
\newcommand{\Supp}{\operatorname{Supp}}
\newcommand{\red}{\operatorname{red}}
\setlength{\parskip}{6pt}
\setlength{\parindent}{0pt}
\newcommand{\forceindent}{\leavevmode{\parindent=15pt\indent}}
\setlength{\topsep}{0.35cm plus 0.3cm minus 0.2cm}

\makeatletter
\def\thm@space@setup{%
  \thm@preskip=0.25cm plus 0.3cm minus 0.16cm
  \thm@postskip=0.01cm plus 0cm minus 0.16cm
}
\makeatother

\begin{abstract}
Let $F$ be the function field of a curve over a complete discretely valued field. Let $\ell$ be a prime not equal to the characteristic of the residue field. Given a finite subgroup $B$ in the $\ell$-torsion subgroup ${}_{\ell}\br(F)$ of the Brauer group, we define the index of $B$ as the minimum of the degrees of field extensions which split all elements in $B$. We give an upper bound for the index of any finite subgroup $B$ in terms of arithmetic invariants of $F$. As a simple application of our result, given a quadratic form $q/F$, where $F$ is the function field of a curve over an $n$-local field, we provide an upper bound to the minimum of degrees of field extensions $L/F$ so that the Witt index of $q\otimes L$ becomes the largest possible.
\end{abstract}
\maketitle

\section{Introduction}
Let $F$ be a field. Recall that the Brauer group $\br(F)$ is a torsion group. The order of an element $\alpha$ in $\br(F)$ is called the period of $\alpha$, denoted by $\per(\alpha)$. The index of $\alpha$, denoted by $\ind(\alpha)$ is the g.c.d.~(or the minimum) of degrees of field extensions $[L:F]$ such that $\alpha \otimes L = 0$. It is a standard fact that $\per(\alpha)$ divides $\ind(\alpha)$ and that they share the same prime factors. Therefore there exists a nonnegative integer $N_{\alpha}$ such that $\ind(\alpha) | [{\per(\alpha)}]^{N_{\alpha}}$.

The period-index problem for Brauer groups asks that when $F$ is the function field of a variety over a global field, local field or an algebraically closed field, given a Brauer class $\alpha$ in the $n$-torsion part of the Brauer group ${}_{n}\br(F)$, can one uniformly bound $N_{\alpha}$ as $\alpha$ varies in ${}_{n}\br(F)$? This bound preferably should not depend on $n$. One may also ask the following variant of the question: 
\begin{ques}
\label{ques1}
Let $F$ be a field of geometric or arithmetic interest (function field of a variety over a global field, local field or an algebraically closed field) and let $n \geq 1$ be an integer. Given any finite subgroup $B \subset {}_{n}\br(F)$, can one find a field extension $L/F$ of uniformly bounded degree (i.e., the degree $[L:F]$ should depend only on $n$ and $F$, and not on the cardinality of $B$) such that $\alpha \otimes L = 0$ for every $\alpha$ in $B$? 
\end{ques}
For number-fields, or rather for global fields, one can answer this question using the Albert, Brauer, Hasse, Noether Theorem and weak approximation. Here, it turns out that for any $B \subset {}_{n}\br(F)$, one can find a field extension of degree $n$ that splits all elements in $B$. 

\begin{defn}
Let $B \subset {}_{n}\br(F)$ be a finite subset. We define the index of $B$, $\ind(B)$ as the minimum of degrees of field extensions $L/F$ such that $\alpha \otimes L = 0$ for every $\alpha$ in $B$.

We define the Generalized Brauer $n$-dimension of $F$, $\sbd_{n}(F)$ as the supremum of $\ind(B)$ as $B$ varies over finite subsets of ${}_{n}\br(L)$, and $L/F$ ranges over finite degree field extensions.
\end{defn}
If $\ \langle B \rangle$ denotes the subgroup generated by $B$, notice that $\ind(B) = \ind(\langle B \rangle)$. Note that one could have replaced {\it minimum} by {\it g.c.d.}~in the definition of $\ind(B)$. However, we do not know whether these two quantities are always equal, and it would be of interest to know this. If indeed they are equal, then by a theorem of Gabber, Liu and Lorenzini (\cite[Theorem 9.2]{GLL13}) applied to the product of Severi-Brauer varieties of each $\alpha$ in $B$, we may just consider separable field extensions in the definition of $\ind(B)$. Furthermore, if these two quantites are equal, the Generalized Brauer $n$-dimension will always be a power of $n$.

Question \ref{ques1}, in other words, asks if $F$ is a field of arithmetic or geometric interest, is $\sbd_{n}(F)$ finite. While for many such fields, this is not known, there is some hope in providing an upperbound for $\sbd_{n}(F)$, when $F$ is a semiglobal field (function fields of curves over complete discretely valued fields). Such fields are amenable to the field patching technique introduced in \cite{HH10}. This technique was developed further in \cite{HHK09}, \cite{HHK15} and \cite{HHK15(1)}. 
\begin{notation}
	\label{notn1}
	Let $R$ be a ring complete with respect to a non-trivial discrete valuation with fraction field $K$ and residue field $k$. Let $F$ be the function field of a geometrically integral curve over $K$. Such fields are called \textbf{semiglobal fields}. For all primes $\ell$ dividing the order of a given finite Galois module, we will always assume that $\ch(k) \neq \ell$. In this case, note that $\ch(K), \ch(F) \neq \ell$ and that $\ell$ is invertible on any model of $F$. 
\end{notation}

 Parimala and Suresh \cite[Theorem 3.6]{PS15} provide the following upperbound to $\sbd_{\ell}(F)$.

\begin{thm}[Parimala, Suresh \cite{PS15}]
	Let $F$, $k$ and $\ell$ be as in Notation \ref{notn1}. Then
	\[\sbd_{\ell}(F) \leq {\ell}^3[\sbd_{\ell}(k(t))]! \cdot [\sbd_{\ell}(k)]!.\]  
\end{thm}
We improve this bound by removing the factorials appearing above.
\begin{thm}[\emph{= Theorem \ref{MAIN1}}]
	\label{Main2prime}
	Let $F$, $k$ and $\ell$ be as above in Notation \ref{notn1}. If $\ell \neq 2$,
	\[\sbd_{\ell}(F) \leq {\ell}^2\sbd_{\ell}(k(t))\sbd_{\ell}(k).\]  
	If $\ell=2$, \[\sbd_{2}(F) \leq 8\sbd_2(k(t))\sbd_2(k).\]
\end{thm}
 As a corollary of Theorem \ref{Main2prime}, we give an upper bound to the generalized Brauer $\ell$-dimension when $F$ is the function field of a curve over an $n$-local field (see Definition \ref{M-LOCAL} below). 
 \begin{defn}
 	\label{M-LOCAL}
 	By a $0$-local field, we mean a finite field. For $n \geq 1$, we say that $F$ is an $n$-local field if it is a complete discretely valued field with residue field an $(n-1)$-local field.
 \end{defn}
 \begin{cor}
 	\label{Main2}
 	Let $F$ be the function field of a curve over an $n$-local field. Let $\ell$ be a prime coprime to the characteristic of the underlying $0$-local field. Then $\sbd_{\ell}(F) \leq {\ell}^{(n^2 +3n)/2}$ for $\ell \neq 2$ and $\sbd_2(F) \leq 2^{(n^2 + 5n-2)/2}$.
 \end{cor}
 \begin{proof}
 	We prove this by induction on $n$. Let $a_n := \sbd_{\ell}(F)$. Let $k$ be the residue field of the field of constants of $F$. Note that $k$ is an $(n-1)$-local field. We denote by $b_{n-1}$ the generalized Brauer dimension of an $(n-1)$ local field.  By Theorem \ref{MAIN1}, we have $a_n \leq {\ell}^2 a_{n-1}b_{n-1}$. Note also, that $b_{n-1} \leq lb_{n-2}$. Therefore $b_{n-1} \leq {\ell}^{n-1}$. Thus one sees that $a_n \leq {\ell}^{(n^2+3n - 4)/2}a_1$. But by Corollary \ref{P-ADIC} since $a_1 \leq \ell^{2}$, one obtains that $a_{n} \leq {\ell}^{(n^{2} + 3n)/2}$.
 	
 	For $\ell = 2$: $a_n \leq 8a_{n-1}b_{n-1} = 8a_{n-1}2^{n-1} = 2^{n+2}$. Therefore one has that $a_n \leq a_12^{(n^2+5n-6)/2} = 2^{(n^2+5n-2)/2}$, again using that $a_1 \leq 2^{2}$ by Corollary \ref{P-ADIC}.  
 \end{proof}
 As another corollary, the bound for the generalized Brauer dimension provides a bound for the index of cohomology classes. This bound is smaller than that stated in \cite[Conjecture 1]{K16} which is not very surprising for fields considered in Corollary \ref{INDEXCOH}. This also shows that the symbol length of any Galois cohomology class defined over such fields is finite in view of \cite[Theorem 4.2]{K16}.
 \begin{cor}
 	\label{INDEXCOH}
 	Let $F$ be the function field of a curve over an $n$-local field, and $m \geq 2$, $j \geq 1$ be integers. Then for any $\alpha$ in $H^{m}(F, \mu_{\ell}^{\otimes j})$, $\ind(\alpha) \leq (\ell -1){\ell}^{(n^2 + 3n)/2}$ for $l\neq 2$. If $\ell=2$, one has $\ind(\alpha) \leq 2^{(n^2 +5n -2)/2}$.
 \end{cor}
 \begin{proof}
 	Let $K = F(\mu_{\ell})$. By the norm residue isomorphism theorem, $\alpha\otimes K = \sum_{i=1}^{m} \beta_i \cup \gamma_i$ where $\beta_i$ are distinct classes in $H^2(K, \mu_{\ell})$ (here we use the fact that $\mu_{\ell}^{\otimes 2} \cong \mu_{\ell}$ over $K$). Note that for $B = \{\beta_1 , \cdots \beta_m \}$, one has $\ind(B) \leq \sbd_{\ell}(F)$. Thus there exists a field extension $M/K$ of degree at most $\ind(B)$ that splits $\alpha$. The result now follows from Corollary \ref{Main2} together with the fact that $[F(\mu_{\ell}):F] \leq \ell-1$.
 \end{proof}
 
 The above bound is not the best. The techniques used to obtain Theorem \ref{Main2prime} also help us in splitting the ``top" cohomology classes much more optimally (see Corollary \ref{top}).
 
We prove Theorem \ref{Main2prime} using the field patching technique. One typically associates to a semiglobal field $F$, overfields coming from a finite collection of closed points and open sets in the special fibre of a two-dimensional normal projective model of $F$. We first follow the strategy of Saltman \cite{S97} to split the ramification of classes in $B$ on a two dimensional projective model of $F$ (see Proposition \ref{RamB}). This allows us to specialize the classes on the special fibre. We then split classes in $B$ in two steps. First, we split them all at once over the function fields of all irreducible components of the special fibre using Lemma \ref{WeakRas}. We then proceed to split them all at once on the remaining closed points using Lemma \ref{WeakRas2}. We have borrowed ideas from \cite{HHKPS17} in the proof of Lemma \ref{WeakRas2}. 

\comm{
\begin{defn}
Let $F$ be a field and $m > 0$ be an integer such that $char(F)$ does not divide $m$. Let $\alpha$ be a class in $H^n(F, \mu_{m}^{\otimes n})$. The symbol length of $\alpha$ denoted by $\lambda_{m}^{n}(\alpha)$ is the smallest integer integer $k$ such that $\alpha$ is expressible as sum of $k$ symbols.

The $\lambda_{m}^{n}$ invariant of $F$, denoted as $\lambda_{m}^{n}(F)$ is the supremum of $\lambda_{m}^{n}(\alpha)$ as $\alpha$ varies over elements in $H^n(F, \mu_m^{\otimes n})$
\end{defn}

The symbol length of cohomology classes is another measure of complexity of fields. It has been studied for Brauer classes by a number of authors over the years. For other Galois cohomology classes, it has been studied by \cite{K16} where it is related to other arithmetic invariants of the field. When the $2$-cohomological dimension of $F$ is finite, the finiteness of $\lambda_{2}^{n}(F)$ is equivalent to the finiteness of the u-invariant $u(F)$. When $F$ is a function fields of a curve over $p$-adic fields, the precise computations, $\lambda_2^{2}(F) = 2$ and $\lambda_2^{1}(F) = 1$ has been used by Parimala and Suresh in \cite{PS07} to show that $u(F) = 8$. 

As an application of the bounds in Theorem \ref{Main2prime} for $\sbd_{\ell}(F)$, we obtain the following:
\begin{cor}
Let $F$ be a field as in Notation \ref{notn1}. Let $n$ be a positive integer. Suppose that $\lambda_{l}^{n}(k)$ and $\lambda_{l}^{n}(k(t))$ are finite. Then $\lambda_l^{n}(F)$ is finite
\end{cor}
\begin{proof}
Let $\alpha$ be an element in $H^n(F, \mu_l^{\otimes n})$. Let us first assume that $\mu_l \subset F$ so that we may identify the Galois modules $\mu_l^{\otimes n}$ and $\mu_l$. By the norm residue isomorphism theorem (\textcolor{red}{Proper reference}), we may write $\alpha = \sum_{i=1}^{k} \beta_i \cup \gamma_i$, where $\beta_i$ are elements in $H^2(F, \mu_l)$ for all $i$. Note that $H^2(F, \mu_l)$ is isomorphic to ${}_{l}\br(F)$. Applying Theorem \ref{Main2prime} to $B = \{ \beta_1, \cdots \beta_k \}$, one concludes that  
\end{proof}

}  
\subsection{Application to splitting quadratic forms}
\begin{defn}
Let $q/F$ be a regular quadratic form of dimension $n$. We define the \textbf{splitting index} $i(q)$ of $q$ to be the minimum of the degrees of field extensions $[L:F]$ such that $q \otimes_{F} L$ has Witt index $\lfloor \frac{n}{2} \rfloor$.

We define the \textbf{splitting dimension of a field} $F$, $i_s(F)$ as the supremum of $i(q)$ as $q$ ranges over quadratic forms over $L$ and $L/F$ ranges over finite degree field extensions.

\end{defn}
\begin{ques}
\label{ques2}
Let $F$ be a field of geometric or arithmetic interest (the fields considered in Question \ref{ques1}). Compute the splitting dimension of $F$ in terms of invariants of $F$.
\end{ques}
Note that if the $u$-invariant of the field is $N$, then an upperbound for the splitting dimension certainly exists. Notice that an anisotropic binary form can be split by a quadratic extension. If one writes an $N$-dimensional anisotropic quadratic form as an orthogonal sum of $\lfloor N/2 \rfloor$ binary forms, then taking the compositum of the quadratic extensions splitting each binary form, one sees that the splitting dimension of such a field $F$ is at most $2^{\lfloor N/2 \rfloor}$. For example, since the $u$ invariant of function fields of curves over $n$-local fields is $2^{n+2}$ \cite[Corollary 4.14 (b)]{HHK09}, a crude bound for their splitting dimension is $2^{2^{(n+1)}}$.

The following related question was asked in an AIM workshop ``Deformation Theory and the Brauer group" in 2011 (see \cite{AimPL11}, there splitting dimension was called the torsion index): If $F$ is a field with finite $u$-invariant $u(F)$. Then is the splitting dimension $i_{s}(F) < 2^{u(F)/2 - 1}$?

A simple application of Corollary \ref{Main2} answers the question raised in the AIM workshop in the positive for somewhat nontrivial classes of fields. Corollary \ref{SPLITTINGQUADRATIC} below gives a sense as to how the splitting dimension grows with the $u$-invariant for function fields of curves over $n$-local fields. We see below the bound is much smaller than the crude bound of $2^{2^{(n+1)}}$.
\begin{cor}
	\label{SPLITTINGQUADRATIC}
Let $F$ be the function field of a curve over an $n$-local field. Assume further that the characteristic of the underlying $0$-local field is not equal to two. Then $i_s(F) \leq 2^{(n^2+5n)/2}$
\end{cor}
\begin{proof}
Let $L/F$ be a finite extension and let $q/L$ be a quadratic form. We may assume that it is even dimensional, for if it is not, then we replace it by a codimension one subform. Let $K/L$ be a quadratic extension splitting its discriminant. Thus one may write $q \otimes K = \sum_{i=1}^{m} \epsilon_i \langle \langle  a_i , b_i \rangle \rangle$ in the Witt group $W(F)$ where $\epsilon_i$ is in $\{ 1, -1 \}$. Applying Corollary \ref{Main2} to the subset $B = \{ (a_1,b_1) , \cdots (a_n,b_n)\} \subset{}_{2}\br(K)$, there exists a field extension of degree $2^{(n^2+5n-2)/2}$ splitting $q_K$. Therefore $q$ is split by an extension of degree $2^{(n^2+5n)/2}$.
\end{proof}

In the final section, using results on the period-index bounds for fields of the form $\mathbb{Q}_p(t)$ due to Saltman \cite{S97} and that of Parimala and Suresh \cite{PS07} on the $u$-invariant, we will show that the splitting dimension of such fields is $8$ (see Proposition \ref{FUB}). This agrees with the bound stated in Corollary \ref{SPLITTINGQUADRATIC} for $n = 1$. However in general, we do not expect the bound in Corollary \ref{SPLITTINGQUADRATIC} to be tight. In Section \ref{LAST}, we record some lower bounds for the splitting dimension.

\begin{remark}
Note that if the splitting dimension of a field $F$ of characteristic unequal to two is finite, there is a bound also on the splitting index of mod-$2$ Galois cohomology classes. If we further assume that the $\cd_{2}(F) < \infty$, then by \cite[Theorem 5.5]{K16}, it follows that $u(F) < \infty$. Thus for fields of finite cohomological dimension (and characteristic unequal to two), the finiteness of the $u$-invariant is equivalent to the finiteness of the splitting dimension.
\end{remark}
\section{Patching Preliminaries}
\label{Patch}
Field patching is a technique developed by Harbater, Hartmann and Krashen in a series of papers to tackle arithmetic questions such as $u$-invariant, period-index problems and local global principles for principal homogenous spaces over semiglobal fields. This will be our principal tool to prove Theorem \ref{MAIN1}. We will briefly describe the patching set-up in this section. We direct the reader to the Luxembourg notes of Harbater (see \cite{H13}) for a reader-friendly account. 
Let $R$ be a complete discretely valued ring with parameter $t$, fraction field $K$ and residue field $k$. Let $F$ be the function field of a geometrically integral curve over $K$. Let $\mathscr{X}/\Spec{R}$ be a normal projective model of the curve. This is a two-dimensional normal scheme with a flat projective morphism to $\Spec{R}$ and whose generic fibre is isomorphic to the curve over $K$.
The fibre over the closed point of $\Spec{R}$ is called the special fibre or the closed fibre of $\mathscr{X}$ and will be denoted by $\mathscr{X}_{k}$. Let $\mathcal{P}$ be a non-empty finite subset of closed points in $\mathscr{X}_{k}$, including the points where irreducible components of $\mathscr{X}_{k}$ meet. Let $\mathcal{U}$ be the set of irreducible components of $\mathscr{X}_{k} \setminus \mathcal{P}$.

To each $U$ in $\mathcal{U}$, consider the following ring $R_{U}$ whose elements are rational functions of $\mathscr{X}$ regular on $U$, or expressed differently:
 $R_U := \cap_{P \in U} O_{\mathscr{X}, P}$. Let $\widehat{R_U}$ be its $t$-adic completion. Note that $\widehat{R_U}$ is a domain (see \cite[Notation 3.3]{HHK09}). Let $F_U$ be its fraction field. Note also that the reduced closed subscheme in $\Spec{\widehat{R_{U}}}$ given by $\sqrt{\langle t \rangle}$ is isomorphic to the reduced induced closed subscheme $U^{\red}$ of $U$. 
 
 To each $P$ in $\mathcal{P}$, let $\widehat{R_P}$ be the completion of $O_{\mathscr{X}, P}$ with respect to its maximal ideal. This also turns out to be a domain. Let $F_P$ be its fraction field.
 
 For a pair $(P, U)$ such that $P \in \overline{U}$ (the closure here is taken in the special fibre $\mathscr{X}_{k}$), let $\wp$ be a height one prime ideal of $\widehat{R_P}$ containing the parameter $t$. We will call such a prime ideal a branch lying on $U$ incident at $P$. Let $\widehat{R_{\wp}}$ be the completion of the localization of $\widehat{R_P}$ at $\wp$ with respect to $\wp$. Note that because $\widehat{R_P}$ is a normal domain, $\widehat{R_{\wp}}$ is a complete discretely valued ring. Let $F_{\wp}$ be its fraction field. There are natural maps from $F_{U}$ and $F_{P}$ to $F_{\wp}$ for a triple $(P, U, \wp)$, where $P \in \overline{U}$ and $\wp$ is a branch lying on $U$ at $P$. The map from $F_P$ to $F_{\wp}$ is induced by the localization map. Let $\eta$ be the prime corresponding to $U^{\red}$ in $R_{U}$. Let $\mathfrak{p} := \wp \cap R_{P}$. Observe that $(R_{U})_{\eta} = (R_{P})_{\mathfrak{p}}$. One has the following chains of inclusions and equality (canonical isomorphisms): $R_{U} \subset (R_{U})_{\eta} = (R_{P})_{\mathfrak{p}} \subset \widehat{R_{\wp}}$. Therefore, the $t$-adic completion of $R_{U}$ i.e $\widehat{R_{U}} \subset \widehat{R_{\wp}}$. 
  
 Intuitively, we think of $\widehat{R_U}$ as a thickening of $U$, $\widehat{R_P}$ as a formal neighbourhood around $P$ and $\widehat{R_{\wp}}$ as the ``overlap" between them. The ``main theorem of patching" essentially says that we can uniquely patch compatible local algebraic structures over $F_U$ and $F_P$ to a global structure over $F$.  For each triple $(P, U, \wp)$ with $P$ in $\mathcal{P}$ and $U$ in $\mathcal{U}$ such that $P \in \overline{U}$ and branch $\wp$ lying on $U$, incident at $P$, the fields $F_{P}$, $F_{U}$ and $F_{\wp}$ forms an inverse factorization system in the terminology of \cite[Definition 2.1]{HHK15}. 
 
Now, consider the triples $(V_{P}, V_{U}, \phi_{\wp})$, where $V_P$ and $V_U$ are finite dimensional vector spaces over $F_P$ and $F_U$ respectively and $\phi_{\wp}: V_{U} \otimes_{F_{U}}F_{\wp} \rightarrow V_P \otimes_{F_P} F_{\wp}$ is an isomorphism for every triple $(P, U, \wp)$ as above. The main theorem of patching (as stated in \cite[Corollary 3.4]{HHK15}) says that there exists a unique (up to isomorphism) vector space $V$ over $F$ such that $V\otimes_{F}F_{\xi} \cong V_{\xi}$ for all $\xi \in \mathcal{P} \cup \mathcal{U}$. This follows from the simultaneous factorization property which holds for $GL_{n}$, i.e., given $A_{\wp}$ in $GL_{n}(F_{\wp})$, there exists matrices $A_{U}$ in $GL_{n}(F_{U})$ and $A_{P}$ in $GL_{n}(F_{P})$ such that $A_{\wp} = A_{U}A_{P}$ for every triple $(P, U, \wp)$.

In fact one can also patch finite dimensional vector spaces with some additional structure such as quadratic forms, central simple algebras, or more importantly for our purpose, separable algebras. In other words, the statement above holds verbatim if the $V_{\xi}$ are separable algebras over $F_{\xi}$ for each $\xi$ in $\mathcal{P} \cup \mathcal{U}$.

Moreover, for connected and rational linear algebraic groups, one also has local-global principles for principal homogenous spaces with respect to these patches (see \cite[Theorem 5.10]{HHK15}). Note that, $PGL_{n}$ is a connected rational group. Recall that isomorphism classes of principal homogenous spaces under $PGL_{n}$ are in a natural bijection with isomorphism classes of central simple algebras of degree $n$. 
 
\begin{thm}[Harbater, Hartmann, Krashen \cite{HHK15}]
\label{HHK}
 Let $A/F$ be a central simple algebra. Let $\mathscr{X}$ be a two-dimensional normal projective model for $F$ as above with special fibre $\mathscr{X}_{k}$. Let $\mathcal{P}$ be a non-empty finite collection of closed points including the points where distinct irreducible components of $\mathscr{X}_{k}$ meet, and let $\mathcal{U}$ be the set of irreducible components of $\mathscr{X}_{k} \setminus \mathcal{P}$. If $A \otimes_{F}F_{\xi}$ is split for all $\xi$ in $\mathcal{P} \cup \mathcal{U}$, then $A/F$ is split.
 \end{thm}
   
\section{Splitting Ramification}
Let $F$ be any field. If $L/F$ is any field extension, then we denote the restriction of a class $\alpha \in H^{n}(F, \mu_{\ell}^{\otimes m})$ to $L$ by $\alpha \otimes L$.

 Let $v$ be a non-trivial discrete valuation of $F$. Let $k(v)$ be the residue field with $\ch(k(v)) \neq \ell$ and let $R_v$ be the discrete valuation ring of $v$. Recall that there exist maps $\partial_v: H^n(F, \mu_l^{\otimes m}) \rightarrow H^{n-1}(k(v), \mu_{\ell}^{\otimes (m-1)})$, called the ramification map. An element $\alpha$ in $H^n(F, \mu_{\ell}^{\otimes m})$ is said to be unramified at $v$ if $\partial_v(\alpha) = 0$. Note that one has the following exact sequence: 
\[ 0 \rightarrow H^n_{\acute{e}t}(R_v, \mu_{\ell}^{\otimes m}) \rightarrow H^n(F, \mu_{\ell}^{\otimes m}) \rightarrow H^{n-1} (k(v) , \mu_{\ell}^{\otimes (m-1)}) \rightarrow 0.  \]
(see for example \cite[Section 3.6]{CT92}). We denote by $\alpha_{k(v)}$, the image of $\alpha$ under the restriction map $H^n(R_v, \mu_{\ell}^{\otimes m}) \rightarrow H^n(k(v), \mu_{\ell}^{\otimes m} )$ and call it the specialization of $\alpha$ at $v$. For any local ring $R$ with residue field $\kappa$, if $\alpha$ is in $H^n_{\acute{e}t}(R, \mu_{\ell}^{\otimes m})$, then we denote the restriction of $\alpha$ in $H^n(\kappa, \mu_{\ell}^{\otimes m})$ by $\alpha_{\kappa}$ and also call it the specialization of $\alpha$.
Let $\beta$ be a class in $H^{n-1}(F, \mu_{\ell}^{\otimes (m-1)})$ unramified at a non-trivial discrete valuation $v$ and consider the element $\alpha = \beta \cup (\pi_v)$ in $H^n(F, \mu_{\ell}^{\otimes m})$. Here $(\pi_v)$ is in $H^1(F, \mu_{\ell})$ and is the class of a parameter for the valuation $v$. Then $\partial_v(\alpha) = \beta_{k(v)}$. 

Suppose that $R_{v} \subset R_{w}$ is a finite extension of discrete valuation rings with respective fraction fields $F \subset L$. Denote the ramification index of $w$ over $v$ by $e_{w/v}$. Then one has the following commutative diagram (See \cite[Proposition 3.3.1]{CT92})
\[
\begin{tikzcd}
H^n(F, \mu_l^{\otimes m})\arrow[r, "\partial_v"]
\arrow[d, "Res_{L/F}"] & H^{n-1}(k(v), \mu_{\ell}^{\otimes (m-1)} ) \arrow[d, "e_{w/v} Res_{k(v)/k(w)}"] \\
H^{n}(L, \mu_{\ell}^{\otimes m})\arrow[r, "\partial_w"]& H^{n-1}(k(w), \mu_{\ell}^{\otimes (m-1)}).
\end{tikzcd}
\]
In particular, if the degree of $L/F$ is prime to $\ell$, then $\partial_{w}(\alpha \otimes L) = 0$ if and only if $\partial_{v}(\alpha) \otimes k(w) = 0$. Moreover, if $L/F$ is totally ramified at $v$ of degree a multiple of $\ell$, then $\partial_{w}(\alpha \otimes L) = 0$. 
\begin{defn}
Let $\mathscr{X}$ be a normal integral scheme, $F$ be its function field and $\ell$ be a prime invertible on $\mathscr{X}$. Let $B = \{ \alpha_1, \cdots \alpha_n \}$ be a finite subset of $H^n(F, \mu_l^{\otimes m})$. We define the ramification divisor of $B$ on $\mathscr{X}$ as the sum of all prime divisors $Y$ on $\mathscr{X}$ for which $\partial_{v_{Y}}(\alpha_i) \neq 0$ for some $i = 1, \cdots n$, where $v_{Y}$ is the discrete valuation corresponding to $Y$.

We say that the ramification of $B$ with respect to a subset $\Omega_{F}$ of non-trivial discrete valuations of $F$ is split if $\partial_v(\alpha) = 0$ for all $\alpha$ in $B$ and for all $v$ in $\Omega_F$.
\end{defn}

We now prove a proposition below which helps us in splitting ramification of a finite collection of cohomology classes in a controlled manner. The idea of the proof is well known and is a minor adaptation of one of Gabber's fix to Saltman's original argument. (see the review of Colliot-Th\'el\`ene in zbMATH of \cite{S97}). For a normal crossing divisor on a regular surface, one can find at most three functions which at every point on the divisor, locally describe it as the zero locus of (appropriate combinations of) any two of those functions. This helps us in splitting the ramification of a finite collection of Brauer classes in a controlled way. We give a proof below for completeness. Lemma \cite[Lemma 2.4.8]{AAIKL17} provides a wider generalization of this idea to any quasi-projective regular scheme of dimension $d \geq 2$.

\begin{prop}
\label{RamB}
 Let $F$ be the function field of $\mathscr{X}/S$, a two-dimensional excellent integral scheme which is projective over some affine scheme $S$ and $\ell$ be a prime invertible on $\mathscr{X}$. Let $B$ be a finite subset of $H^n(F, \mu_{\ell}^{\otimes m})$. If $\ell \neq 2$, then there exists an extension of degree $\ell^2$ which splits the ramification of $B$ with respect to all discrete valuations with centers on $\mathscr{X}$. If $\ell = 2$, there exists a degree $8$ extension which splits the ramification of $B$ with respect to all discrete valuations with centers on $\mathscr{X}$. 
\end{prop}
\begin{proof}
 Let $K = F(\mu_{\ell})$, $v$ be a nontrivial discrete valuation of $K$ and let $w$ be the restriction of $v$ to $F$. The ramification index and the degree of the residue field extension is coprime to $\ell$. Let $\alpha$ be in $B$. Notice that $\partial_v(\alpha \otimes K) = 0$ if and only if $\partial_w(\alpha) = 0$. Conversely if $w$ is a discrete valuation of $F$ and $v$ is any extension of $w$ in $K$, one has $\partial_v(\alpha_K) = 0$ if and only if $\partial_w(\alpha) = 0$. We may thus assume without loss of generality that $\mu_{\ell} \subset F$.

 Using Lipman's Theorem on embedded resolution of singularities (see \cite[Section 9.2.4, Theorem 2.26]{Li02}), one may assume that there exists a regular projective model $\mathscr{Y}$ of $F$ such that the ramification divisor of $B$ is a normal crossing divisor $D$, and that $\Supp(D) = \Supp(D_1) \cup \Supp(D_2)$, where $\Supp(D_1)$ and $\Supp(D_2)$ are regular, but not necessarily connected curves.
 
Let $\mathcal{P}$ be a finite set of points on each irreducible component of $\Supp(D_1) \cup \Supp(D_2)$ (at least one in every component), including the intersection points of $D_1$ and $D_2$. Since $\mathscr{X}$ is projective over an affine scheme $S$, one may find an affine open set $U$ containing all the points in $\mathcal{P}$ \cite[Chapter 3, Proposition 3.36]{Li02}. Let $A$ be the semi-local ring obtained by semi-localizing $U$ at the points in $\mathcal{P}$. Since $A$ is semi-local, $\Pic(A)$ is trivial.

We consider the case $\ell \neq 2$ first. Since $\Pic(A)$ is trivial, there exists a rational function $f_1$ such that $\Divisor_{A}(f_1) = D_1 + 2D_2$ (Here $D_1 + 2D_2$ is viewed as a divisor on $\Spec{A}$.). Therefore $\Divisor_{\mathscr{X}}(f_1) = D_1 + 2D_2 + E$ where $E$ does not pass through any points in $\mathcal{P}$. Let $\mathcal{P}_1$ be the points of intersection of $\Supp(E)$ with $\Supp(D_1) \cup \Supp(D_2)$. Let $A_1$ be the semi-local ring at the points $\mathcal{P} \cup \mathcal{P}_1$. Again since $\Pic(A_1)$ is trivial, one may find a rational function $f_2$ such that $\Divisor_{A_1}(f_2) = D_1 +D_2$. Therefore $\Divisor_{\mathscr{X}}(f_2) = D_1 +D_2 + G$ where $G$ does not pass through any points in $\mathcal{P} \cup \mathcal{P}_1$.

Consider the field extension $L = F(\sqrt[\ell]{f_1}, \sqrt[\ell]{f_2})$. Let $v$ be a discrete valuation of $L$ with center $x$ on $\mathscr{X}$. We may assume that $x$ lies on the ramification divisor of $B$. Let $\alpha$ be any non-trivial element in $B$. Note that it suffices to show that the residue of $\alpha$ at each height one prime ideal of $O_{\mathscr{X},x}$ is split by $L$. 

Suppose $x$ is a point of codimension at most two, lying only on one of $\Supp(D_i)$ and not on $\Supp(E)$ or $\Supp(G)$. Observe that on $O_{\mathscr{X}, x}$, the local equation for $D_{i}$ is given by $f_{2}$. We may express $\alpha = \alpha_0 + \beta_1 \cup(f_{2})$, where $\alpha_0$ and $\beta_1$ are unramified on $O_{\mathscr{X},x}$. The subextension $F(\sqrt[\ell]{f_2})$ is totally ramified on the ramification locus of $\alpha$ in $O_{\mathscr{X},x}$. Therefore by \cite[Proposition 3.3.1]{CT92}, we see that $\partial_v(\alpha \otimes L) = 0$. 

If $x$ lies on $\Supp(D_1) \cap \Supp(D_2)$, then on $O_{\mathscr{X},x}$, the equation for $D_1$ is given by $f_2^2/f_1$ and for $D_2$ is given by $f_1/f_2$. Since $L = F(\sqrt[\ell]{f_2^2/f_1}, \sqrt[\ell]{f_1/f_2})$ totally ramifies the local parameters for $D_1$ and $D_2$ on $O_{\mathscr{X},x}$, we are done. Similarly, if $x$ lies on $\Supp(D_1) \cap \Supp(G)$, $D_1$ is also given by $f_1$. The subextension $F(\sqrt[\ell]{f_1})/F$ splits the ramification. If $x$ lies on $\Supp(D_2) \cap \Supp(G)$, the local equation for $2D_2$ on $O_{\mathscr{X}, x}$ is given by $f_1$. But because $\ell \neq 2$, the local parameter for $D_2$ is totally ramified in the extension $F(\sqrt[\ell]{f_1})/F$.
Finally, if $x$ lies on $\Supp(D_2) \cap \Supp(E)$ or on $\Supp(D_1)\cap \Supp(E)$, $F(\sqrt[\ell]{f_2})/F$ splits the ramification. 

We now come to the case $\ell =2$. Consider three functions $f_1$, $f_2$ and $f_3$ such that $\Divisor_{\mathscr{X}}(f_1) = D_1 + D_2 + E$, $\Divisor_{\mathscr{X}}(f_2) = D_1 + G$ and $\Divisor_{\mathscr{X}}(f_3) = D_2 + H$ where the supports of no three divisors among $D_1$, $D_2$, $E$, $G$ and $H$ intersect. This can be arranged using a similar argument as in the previous case; choose $f_1$ by semi-localizing at the intersection points $\mathcal{P}$ of $\Supp(D_1)$ and $\Supp(D_2)$ and containing at least one point on every irreducible component of $\Supp(D_1)$ and $\Supp(D_2)$. Choose $f_2$ by semi-localizing at $\mathcal{P} \cup \mathcal{P}_1$, where $\mathcal{P}_1$ is the set of intersection points of $\Supp(E)$ and $\Supp(D_1) \cup \Supp(D_2)$. Repeat the process to choose $f_3$. Consider the extension $L = F(\sqrt{f_1}, \sqrt{f_2}, \sqrt{f_3})$. Let $v$ be a discrete valuation of $L$ with center $x$ on $\mathscr{X}$. Just as in the previous case, we show that the local parameters for $D_i$ on $O_{\mathscr{X}, x}$ are totally ramified by $L/F$. We will illustrate this in only two cases, the rest being similar. First, if $x$ lies in $\Supp(D_1) \cap \Supp(D_2)$, the local equations for $D_1$ and $D_2$ are given by $f_2$ and $f_3$ respectively. Thus the subextension $F(\sqrt{f_2}, \sqrt{f_3})/F$ splits the ramification. If $x$ lies on $\Supp(D_1) \cap \Supp(F)$, the extension $F(\sqrt{f_1})$ splits the ramification.   
\end{proof}

As a consequence of Proposition \ref{RamB}, we can simultaneously split a finite collection of Brauer classes of prime period $\ell \neq p$ on the function field of a $p$-adic curve by making a degree $\ell^2$ extension for $\ell \neq 2$. This follows from \cite{S97}, \cite{S98}. In the case of $p$-adic curves, the bound for $\ell=2$ is in fact $4$ and not $8$ as Proposition \ref{RamB} suggests. An argument of a different flavor which works in the case of $p$-adic curves is given in the review of Colliot-Th\'el\`ene of \cite{S97} where it is shown that for any $\ell$, a degree ${\ell}^2$ extension suffices to split the ramification. For completeness, we record it below:
\begin{cor}
	\label{P-ADIC}
Let $F$ be the function field of $\mathscr{X}/S$, a two-dimensional excellent integral scheme which is projective over an affine scheme $S$. Let $\ell$ be a prime invertible on $\mathscr{X}$. Assume further that the residue field of every closed point of $\mathscr{X}$ is a finite field. Let $B \subset {}_{\ell}\br(F)$ be a finite subset. There exists an extension of degree $\ell^2$ that splits the ramification of $B$ with respect to all discrete valuations with centers on $\mathscr{X}$. In particular, if $F$ is the function field of a curve over a nonarchimedean local field, and $\ell$ is a prime unequal to the characteristic of the underlying residue field, then $\sbd_{\ell}(F) \leq {\ell}^2$.
\end{cor}
\begin{proof}
As in the proof of Proposition \ref{RamB}, we may assume without loss of generality that $F$ contains $\mu_{\ell}$.
Using Lipman's resolution of singularities, we may assume that the ramification divisor of $B$ on some regular projective model $\mathscr{X}$ (abusing notation) of $F$ is a normal crossing divisor expressible as a union of two regular divisors $D_1$ and $D_2$. By a semi-local argument used as in the proof of Proposition \ref{RamB}, we may find a function $f_1$ in $F$ such that $\Divisor_{\mathscr{X}}(f_1) = D_1 + D_2 + E$. Note here that $\Supp(E)$ does not pass through any of the intersection points of $\Supp(D_1)$ and $\Supp(D_2)$. Let $\mathcal{P}$ be the intersection points of $\Supp(D_1)$, $\Supp(D_2)$ and $\Supp(E)$. Semi-localizing at $\mathcal{P}$, we may find $f_2$ in $F$ such that $\Divisor_{\mathscr{X}}(f_2) = D_1 + G$ and $f_2$ is a unit at closed points in $\Supp(D_2) \cap \Supp(E)$. By the Chinese Remainder Theorem, there exists an element $g$ in the above semi-local ring which is a unit at all points in $\mathcal{P}$ and such that $f_2g$ is not an $\ell^{th}$ power in $k(x)$ for every closed point $x$ in $\Supp(D_2) \cap \Supp(E)$. We will abuse notation and denote $f_2g$ by $f_2$. We thus assume that $\Divisor_{\mathscr{X}}(f_2) = D_1 + G$ and $f_2$ is not an $\ell^{th}$ power at every closed point in $\Supp(D_2) \cap \Supp(E)$.

We claim that $L = F(\sqrt[\ell]{f_1}, \sqrt[\ell]{f_2})$ splits the ramification of $B$. Let $v$ be a discrete valuation of $L$ with center $x$ on $\mathscr{X}$. The only case that we consider here is when $x$ lies on $\Supp(D_2) \cap \Supp(E)$. The rest are similar to the cases considered in the proof of Proposition \ref{RamB}. Now let $\alpha$ be in $B$. We may express $\alpha = \alpha_0 + (u, \pi_2)$, where $\alpha_0$ is unramified on $O_{\mathscr{X}, x}$ and $\pi_2$ is a local parameter of $D_2$. On $O_{\mathscr{X},x}$, we may write $f_1 = w \pi_2 \delta$, where $w$ is a unit and $\delta$ is a local parameter for $E$. We may therefore rewrite $\alpha = \alpha_0^{\prime} + (u , \delta^{-1}) + (u , f_1)$ for some unramified class $\alpha_0^{\prime}$ and unit unit $u$. Restricting to $K = F(\sqrt[\ell]{f_1})$, we obtain $\alpha \otimes K = \alpha_0^{\prime}\otimes K + (u, \delta^{-1}) \otimes K$. Note that $\partial_v(\alpha \otimes L) = (\bar{u})^{-v(\delta)}$, where $(\bar{u})$ is an $\ell^{th}$ power class in the residue field $k(v)$. Let $\mathscr{Y} \rightarrow \mathscr{X}$ be the normalization of $\mathscr{X}$ in $L$ and let $y$ be the center of $v$ on $\mathscr{Y}$. 

Note first that $k(x)$ is a finite field since $x$ is a closed point. Second, since $f_2$ is not an $\ell^{th}$ power in $k(x)$, the subextension $F(\sqrt[\ell]{f_2})/F$ induces a non-split unramified extension of $k(x)$. Therefore the residue field extension of $k(y)/k(x)$ contains a degree $\ell$ extension. As a result $\bar{u}$ in $k(x)$ becomes an $\ell^{th}$ power in $k(y)$. Since $k(y) \subset k(v)$, $\bar{u}$ becomes an $\ell^{th}$ power in $k(v)$. Thus, $\partial_v(\alpha \otimes L) = 0$.

Now let $F$ be the function field of a curve over a non-archimedean local field $K$. Applying the above to a regular projective model $\mathscr{X}/\Spec{O_{K}}$, and using the fact that Brauer classes unramified at discrete valuations with centers on $\mathscr{X}$ are trivial, one obtains that any finite collection $B$ of $\ell$-torsion Brauer classes can be split by an extension of degree at most $\ell^{2}$.
\end{proof}
Proposition \ref{RamB} plays a role towards our goal of splitting Brauer classes as economically as we can. Once the ramification is split, one may specialize cohomology classes on the special fiber of some regular projective model. The function fields of irreducible components of the special fiber are fields of ``lower complexity". As an application of Proposition \ref{RamB}, we record the following observation:
\begin{cor}
\label{top}
In the situation of Notation $\ref{notn1}$, suppose also that $\cd_{\ell}(k) \leq n-2$, where $n \geq 2$. Let $B$ be a finite subset of $H^n(F, \mu_{\ell}^{\otimes m})$. If $\ell \neq 2$, there exists a field extension of degree at most $\ell^2$ which splits all elements in $B$. If $\ell=2$, there exists a field extension of degree $8$ splitting $B$.
\end{cor}
\begin{proof}
Let $\mathscr{X}$ be a regular projective model of $F$ such that the ramification divisor of $B$ is a strict normal crossing divisor. Let $L/F$ be a field extension as in Proposition \ref{RamB}, and let $\mathscr{Y}$ be a regular projective model of $L$. Let $\{Y_i\}$ be the collection of irreducible components of the special fiber of $\mathscr{Y}$ with respective generic points $\{ \eta_i\}$. Let $\alpha$ be a non-trivial element of $B$. Since $\alpha \otimes L$ is unramified at every discrete valuation of $L$, we may specialize it to the function field $k(Y_i)$. Since $\cd_{\ell}(k) \leq n -2$, one has $\cd_{\ell}(k(Y_{i})) \leq n-1$. Therefore $\alpha_{k(Y_i)} = 0$ and so at the completion $L_{\eta_{i}}$ of $L$ we have $\alpha \otimes L_{\eta_i} = 0$. By \cite[Proposition 3.2.2]{HHK14}, there exists a dense open affine $U_i$ of $Y_i$ which does not contain points of any other irreducible component such that $\alpha \otimes L_{U_i} = 0$. Let $\mathcal{P}$ be the complement of $\cup U_i$ in the special fiber of $\mathscr{Y}$. For every $P$ in $\mathcal{P}$, $\alpha \otimes L_{P}$ is unramified at every height one prime of $\widehat{R_{P}}$. Thus by \cite[Theorem 2.1]{SM19}, $\alpha \otimes L_{P}$ comes from $H^n_{\acute{e}t}(\widehat{R_P}, \mu_{l}^{\otimes m})$. We may specialize $\alpha$ at the residue field $k(P)$ of $\widehat{R_P}$. Since $\cd_{\ell}(k) \leq n-2$, we have $\alpha_{k(P)} = 0$ Thus by \cite[Pg 224, Corollary 2.7]{Mi80}, $\alpha \otimes L_{P} = 0$. Using \cite[Corollary 3.1.6]{HHK14}, we conclude that $\alpha \otimes L = 0$.
\end{proof}

\section{Main Result}
The next two lemmas allow us to construct a global extension of $F$ from extensions $L_{\xi}/F_{\xi}$ for suitable patches $\xi$ in $\mathcal{P} \cup \mathcal{U}$. Lemma \ref{WeakRas} is well known and Lemma \ref{WeakRas2} is a variant of \cite[Theorem 2.6]{HHKPS17} and ideas used in the proof of \cite[Proposition 2.5]{HHKPS17}. However, we do not need the hypothesis that $\ch(k) = 0$ needed there, since we work only with unramified extensions. Working with unramified extensions also has the advantage that we also do not require the points $P$ in Lemma \ref{WeakRas2} to be unibranched.
\begin{lemma}
\label{WeakRas}
Let $\mathscr{X}/\Spec{R}$ be a regular model of a semiglobal field $F$ (see Notation \ref{notn1}). Let $\{X_1, \cdots X_n \}$ be the irreducible components of the special fibre $\mathscr{X}_{k}$. Let $\eta_i$ be the generic points of $X_i$ and $F_{\eta_{i}}$ denote the completion of $F$ at the discrete valuation given by $\eta_i$. Suppose that $L_{\eta_{i}}/F_{\eta_{i}}$ are separable field extensions of degree $d$. Then there exists a field extension $L/F$ of degree $d$ such that $L \otimes_{F}F_{\eta_{i}} \cong L_{\eta_{i}}$.
\end{lemma}
\begin{proof}
Since $L_{\eta_{i}}/F_{\eta_{i}}$ are separable, $L_{\eta_{i}} \cong F_{\eta_{i}}[x]/ \langle f_{\eta_{i}}(x) \rangle$. By weak approximation we may find a polynomial $f(x)$ in $F[x]$ sufficiently close to $f_{\eta_{i}}(x)$ so that by Krasner's lemma we have $L_{\eta_{i}} \cong F[x]/\langle f(x) \rangle \otimes_F F_{\eta_{i}} $. Taking $L = F[x]/ \langle f(x) \rangle$, the claim follows.
\end{proof}
\begin{lemma}
\label{WeakRas2}
 Let $\mathscr{X}/\Spec{R}$ be a normal projective model of a semiglobal field $F$ (see Notation \ref{notn1}). Let $\mathcal{P}$ be a finite non-empty set of closed points on the special fibre $\mathscr{X}_{k}$ of $\mathscr{X}$ which includes points where irreducible components of $\mathscr{X}_{k}$ meet. For each point $P$ in $\mathcal{P}$, let $l(P)/k(P)$ be finite separable extensions of the residue fields $k(P)$ of degree $d$. Let $L_P/F_P$ be their unramified lifts. Then there exists a field extension $L/F$ of degree $d$ such that $L \otimes_{F} F_{P} \cong L_{P}$.
\end{lemma}
\begin{proof}
For $P \in \mathcal{P}$, let $\wp$ be a branch incident at $P$. Observe that $L_P\otimes_{F_P}F_{\wp} \cong \prod_{i}L_{\wp_i}$ where $L_{\wp_i}/F_{\wp}$ are finite unramified field extensions of $F_{\wp}$. Let $l(\wp_i)/k(\wp)$ be the corresponding residue field extensions. Set $l(\wp) := \prod_{i}l(\wp_i)$, $L_{\wp} := \prod L_{\wp_i}$. Going by our notation, we have that $L_{P}\otimes_{F_P}F_{\wp} \cong L_{\wp}$. Thus associated to every branch $\wp$ on the special fiber of $\mathscr{X}$, we obtain the separable algebras $L_{\wp}/F_{\wp}$ and $l(\wp)/k(\wp)$. To obtain a global extension $L/F$, we need to construct separable algebras $L_{V}/F_{U}$ for suitable open sets $U$ in the complement of $\mathcal{P}$ in the special fiber such that $L_{V} \otimes_{F_{U}}F_{\wp} \cong L_{P} \otimes_{F_{P}}F_{\wp}$ for every triple $(P, U, \wp)$.

Let $\mathcal{U}$ be the set of irreducible components of the the complement of $\mathcal{P}$ on the special fibre. For each $U$ in $\mathcal{U}$, let $\mathcal{B}_{U}$ denote the set of branches lying on $U$. Note that the function field $k(U)$ is dense in $\prod_{\wp \in \mathcal{B}_{U}} k(\wp)$ for every $\wp$ in $\mathcal{B}_{U}$. Using Krasner's lemma and weak approximation, there exists a separable algebra $l_{U}/k(U)$ such that $l_{U}\otimes_{k(U)} k(\wp) \cong l(\wp)$ for every branch $\wp$ in $\mathcal{B}_{U}$. 

Let $V \rightarrow U$ be the normalization of $U$ in $l_{U}$. After shrinking $U$ if necessary, we may assume that the map is \' etale. We abuse notation and denote this new open set by $U$. Let $ \mathcal{P} \cup \mathcal{P}_1$ be the complement of these new open sets on the special fiber $\mathscr{X}_{k}$.  By \cite[Corollaire 8.4]{GR71}, we may uniquely lift $V \rightarrow U$ to get an \'etale algebra $\widehat{S_{V}} /\widehat{R_{U}} $. Note that $\widehat{S_{V}}$ is a product of domains. Let $L_{V}$ be the product of their fraction fields. We claim that $L_{V} \otimes F_{\wp} \cong L_{\wp}$ for each triple $(U, P, \wp)$. 

Note that for every branch $\wp$ incident at $P$ in $\mathcal{P}$, we have the following sequence of isomorphisms: $\widehat{S_{V}} \otimes_{\widehat{R_{U}}}\widehat{R_{\wp}}\otimes_{\widehat{R_{\wp}}} \widehat{R_{\wp}}/ \wp  \cong \widehat{S_{V}} \otimes_{\widehat{R_{U}}}\widehat{R_{U}}\otimes_{\widehat{R_{U}}} k(\wp) \cong \widehat{S_{V}} \otimes_{\widehat{R_{U}}}k(U)\otimes_{k(U)} k(\wp) \cong l_{U}\otimes_{k(U)} k(\wp) \cong l(\wp)$. Let $\widehat{S_{\wp}}$ be the integral closure of $\widehat{R_{\wp}}$ in $L_{\wp}$. By \cite[Theorem 6.1]{GR71}, we have that $\widehat{S_{V}}\otimes_{\widehat{R_{U}}} \widehat{R_{\wp}} \cong \widehat{S_{\wp}} $ and therefore $L_{V}\otimes_{F_{U}}F_{\wp} \cong L_{\wp}$. 

The algebras $L_{V}/F_{U}$ induce algebras of the same dimension on the branches incident at the  points in $\mathcal{P}_1$. Using weak approximation again, there exist compatible algebras at these points. All this patches together to give an algebra $L/F$ such that $L\otimes_{F}F_{P} \cong L_{P}$. Since $L_P/F_P$ is a field extension of degree $d$, so is $L/F$.  
\end{proof}

\begin{remark}
The lemma shows that \textit{unramified} extensions at closed points of a projective normal model induce a global extension of the function field. Note that in general, arbitrary extensions of closed points do not induce global extensions as remarked in \cite[Remark 2.7(b)]{HHKPS17}.
\end{remark}
\comm{
 
\begin{proof}
Blow up at the attaching points and get another model of $F$. Abusing notation, we call this model $\mathcal{X}$. Since $U_i$ does not contain the attaching points, they are not changed, neither is the field extension $L_{U_i}/F_{U_i}$. Let $E_1$ and $E_2$ be the exceptional curves attached to $X_i$ and $P_1$ and $P_2$ be the intersection points of $E_1$ and $E_2$ with $X_i$ respectively. 
Denote by $\wp_{i1}$ and $\wp_{i2}$ the branches along $U_i$ at $P_1$ and $P_2$ respectively and by $\wp_1$ and $\wp_2$ the branches along $E_1$ and $E_2$ at $P_1$ and $P_2$ respectively. For $j=1,2$, suppose that $L_{U_i}\otimes_{F_{U_i}}F_{\wp_{ij}} \cong \prod_{k} E_{\wp_{ij}}^{(k)}$ where $E_{\wp_{ij}}^{(k)}/F_{\wp_{ij}}$ is a finite separable field extension. Using \cite[Proposition 2.3]{CHHKPS17}, there exist separable field extension $L_{P_j}^{(k)}/F_{P_j}$ such that $L_{P_j}^{(k)}\otimes_{F_{P_j}} F_{\wp_{ij}} \cong E_{\wp_{ij}}^{(k)}$ and $L_{P_j}^{(k)} \otimes_{F_{P_j}} F_{\wp_j}$ is a split algebra. Let $V_1 \subset E_1$ and $V_2 \subset E_2$ be the complement of the attaching points of $E_1$ and $E_2$. Let $L_{V_{i}}/F_{V_{i}}$ be the split algebra of the same degree. By construction, these \textcolor{red}{accent}etale algebras patch together to give an algebra $L/F$ inducing them of the same degree and which is necessarily a field extension since $L_{U_{i}}/F_{U_{i}}$ are field extensions. 
\end{proof}
}
We are now in a position to prove our main theorem. Lemma \ref{WeakRas} and Lemma \ref{WeakRas2} allows us to construct field extensions which split $B$, first generically on the special fiber and then on the remaining closed points.
\begin{thm}
\label{MAIN1}
Let $F$, $\ell$ and $k$ be as above in Notation \ref{notn1}. If $\ell \neq 2$,
 \[\sbd_{\ell}(F) \leq {\ell}^2\sbd_{\ell}(k(t))\sbd_{\ell}(k).\]  
If $\ell=2$, \[\sbd_{2}(F) \leq 8\sbd_2(k(t))\sbd_2(k).\]
\end{thm}
\begin{proof}
If $\ell \neq 2$, let $K/F$ be a degree $\ell^2$ extension that splits the ramification of $B$ with respect to discrete valuations of $F$ centered on a two-dimensional regular projective model, as chosen in Proposition \ref{RamB} or if $\ell=2$, let $K/F$ be a degree $8$ extension as chosen in Proposition \ref{RamB}. Let $\mathscr{X}$ be a regular projective model of $K$ such that $\{ X_1, \cdots X_n \}$ are the irreducible components of the special fibre $\mathscr{X}_{k}$.
 
We prove the statement in two steps. First we show that there exists a field extension $M/K$ of degree $\sbd_{\ell}(k(t))$ with some normal model $\mathscr{Y}$ such that $B$ is split on all but finitely many closed points of the special fibre of $\mathscr{Y}$. After that, in Step 2, we construct another extension $L/M$ of degree $\sbd_{\ell}(k)$ which splits $B$.

\textbf{Step 1:} Let $\eta_{i}$ be the generic points of $X_i$ and let $K_{\eta_{i}}$ denote the completion of $K$ at $\eta_i$. Since $B$ is unramified, we may specialize every element $\alpha$ in $B$ to the residue field $k(X_i)$ of $K_{\eta_i}$. There exists a separable field extension $m_i/k(X_i)$ of degree at most $\sbd_{\ell}(k(t))$ splitting $\alpha_{k(X_i)}$. We may as well assume that $m_i/k(X_i)$ has degree $\sbd_{\ell}(k(t))$. Let $M_i/K_{\eta_{i}}$ denote the unramified lifts of the extensions $m_i/k(X_i)$. Since $\alpha_{m_{i}}$ is split, $\alpha \otimes M_i$ is also split. By Lemma \ref{WeakRas}, there exists a field extension $M/K$ of degree $\sbd_{\ell}(k(t))$ such that $M\otimes_{K}K_{\eta_i} \cong M_i$. Let $f :\mathscr{Y} \rightarrow \mathscr{X}$ be the normalization of $\mathscr{X}$ in $M$. In view of our choice of $M/K$, note that for $i = 1, \cdots n$, $\eta_i^{\prime} := f^{-1}(\eta_i)$ are the generic points of $Y_i := f^{-1}(X_i)$, the irreducible components of the special fibre of $\mathscr{Y}$ and $M_i$ are the respective completions of $M$ at those points. Since $\alpha \otimes M_{i}$ is split, by \cite[Proposition 5.8]{HHK15}, there exist non-empty dense affine open subsets $U_i \subset Y_i$ which do not meet any other component and such that $\alpha \otimes M_{U_{i}}$ is split. Thus, $\alpha$ is split everywhere on the special fibre except possibly at the complement of the open sets $U_i$.

\textbf{Step 2}: Let $\mathcal{U}$ be the set of the open sets $U_i$ from Step 1. Let $\mathcal{P}$ be the complement of $\cup_{i} U_i$ on the special fibre $Y$ of $\mathscr{Y}$. Note that $\alpha \otimes M_{P}$ comes from a class on $\widehat{O_{\mathscr{Y},P}}$ for all points $P$ in $\mathcal{P}$. This is because $\alpha$ is unramified on the regular local ring $\widehat{O_{\mathscr{X}, f(P)}}$. We may therefore specialize it to the residue field $k(P)$ of $\widehat{O_{\mathscr{Y}, P}}$. Let $l(P)/k(P)$ be separable field extensions of degree $\sbd_{\ell}(k)$ splitting $\alpha_{k(P)}$. Let $L_P/M_P$ be the lift. By \cite[Pg 224, Corollary 2.7]{Mi80}, $\alpha \otimes L_{P}$ is split. By Lemma \ref{WeakRas2}, there exists a field extension $L/M$ of degree $\sbd_{\ell}(k)$ inducing $l(P)/k(P)$. 

We claim that $\alpha \otimes L$ is split. Let $g: \mathscr{Z} \rightarrow \mathscr{Y}$ be the normalization of $\mathscr{Y}$ in $L$. Let $\mathcal{P}^{\prime}$ be the inverse images of the points $P$ in $\mathcal{P}$ under the normalization map $g$. Let $\mathcal{U}^{\prime}$ be the set of irreducible components of the complement of $\mathcal{P}^{\prime}$ in the special fibre $Z$. Note that for each $U^{\prime}$ in $\mathcal{U}^{\prime}$, there exists some $U_i$ in $\mathcal{U}$ such that $ M_{U_{i}} \subset L_{U^{\prime}}$. Since $\alpha \otimes M_{U_i}$ is split, so is $\alpha \otimes L_{U^{\prime}}$. Furthermore, for each $P^{\prime}$ in $\mathcal{P}^{\prime}$, $L_{P^{\prime}}$ is an unramified extension of $F_P$ for some $P$ induced by the residue field extension $l(P)/k(P)$ which splits $\alpha_{k(P)}$. Thus $\alpha \otimes L_{P^{\prime}}$ is split. By \cite[Proposition 6.3]{HH10}, $(\mathcal{U^{\prime}}, \mathcal{P^{\prime}})$ forms an inverse factorization system. By Theorem \ref{HHK}, it follows that $\alpha \otimes L$ is split for every $\alpha$ in $B$.
\end{proof}

\comm{In Step 1 we are assuming the fact that there is a one to one correspondence between the irreducible components in $\mathcal{Y}$ and $\mathcal{X}$. Why should this be true? This is just by construction. Let $\eta_i$ be the generic point downstairs. The possible extensions of the valuation given by this generic point in $L$ are in one to one correspondence with the maximal ideals in $L \otimes_F \widehat{F_{\eta_i}}$. By construction, there is only one. We most definitely need a citation above. }

\begin{cor}
\label{CORMAIN}
Let $F$, $\ell$ and $k$ be as in Notation \ref{notn1}. If $\ell \neq 2$,
\[ \sbd_{\ell^m}(F) \leq ({\ell^2})^m [\sbd_{\ell}(k(t))]^{m} [\sbd_{\ell}(k)]^{m}. \]
If $\ell = 2$, \[ \sbd_{2^m}(F) \leq 8^m [\sbd_{2}(k(t))]^m [\sbd_{2}(k)]^m. \]
\end{cor}
\begin{proof}
We show this by induction on $m$. We show this only for $\ell \neq 2$ since the case for $\ell = 2$ is similar. The base case, $m = 1$, follows from Theorem \ref{MAIN1}. Suppose that the statement holds for $m-1$. Let $B$ be a finite subset in ${}_{\ell^{m}}\br(F)$. We denote by $\ell^{m-1}B$ the set in ${}_{\ell}\br(F)$ obtained by multiplying each element in $B$ by $\ell^{m-1}$. If $L/F$ is a field extension, we denote by $B_{L}$ the subset in ${}_{\ell^m}\br(L)$ obtained by restricting each element in $B$ to the field $L$. By Theorem \ref{MAIN1}, it follows that there exists a field extension $L/F$ of degree at most ${\ell}^2\sbd_{\ell}(k(t))\sbd_{\ell}(k)$ splitting $\ell^{m-1}B$. Thus, $B_{L}$ is a subset in ${}_{\ell^{m-1}}\br(L)$. By induction, there exists a field extension $M/L$ of degree at most $({\ell^2})^{m-1} [\sbd_{\ell}(k(t))]^{m-1} [\sbd_{\ell}(k)]^{m-1}$ splitting $B_{L}$. Thus, the extension $M/F$ splits all elements in $B$ and the result follows.
\end{proof}
\section{Lower bounds for splitting dimension}
\label{LAST}

We provide lower bounds for the splitting dimension for some fields in this section. For fields satisfying the hypothesis of Proposition \ref{FUB} below (examples include $\mathbb{Q}_p(t)$, $\mathbb{F}_p(x, y)$, $\mathbb{F}_p((x, y))$), using the lower bound in Proposition \ref{CDV} and the upperbound in Proposition \ref{FUB}, we can compute the splitting dimension of certain function fields of curves over complete discretely valued fields. For function fields of varieties over an algebraically closed fields we also record a lower bound for the splitting dimension in Proposition \ref{dDIM}.

If $F$ is a totally imaginary number field or a function field of a curve over a finite field, it is well known that $u(F) = 4$. Writing an anisotropic four dimensional quadratic form as a sum of two binary forms, we see that $i_s(F) \leq 4$. To construct a quadratic form with splitting index at least $4$, we use the Albert Brauer Hasse Noether Theorem. Let $\alpha$ be a non-trivial element in ${}_{2}\br(F)$. Let $v$ be a place where $\alpha$ is ramified. By Weak approximation and Krasner's Lemma, there exists a quadratic field extension $L/F$ such that $v$ is totally split in $L$. Note therefore that $\alpha \otimes L$ is non-trivial. Let $p$ be the norm form of $L/F$ and $p_{\alpha}$ be that of the quaternion algebra determined by the class $\alpha$. Consider the quadratic form $q = p \perp p_{\alpha}$. If $M/F$ splits $q$, $M$ contains $L$. Since $M$ splits $q\otimes L$, $2$ divides $[M:L]$. Therefore $[M:F] \geq 4$. This shows that $i_s(F) = 4$. We will adopt the same strategy to give the lower bounds in Propositions \ref{dDIM} and \ref{CDV}.

\begin{prop}
\label{FEXAMPLE}
Let $K$ be a complete discretely valued field with parameter $t$ and residue field $k$ of characteristic not equal to $2$. Then $i_s(K) \leq 2i_s(k)$.
\end{prop}
\begin{proof}
Let $q / K$ be a quadratic form. Then $q = q_1 \perp tq_2$, where the entries of $q_1$ and $q_2$ are units in its ring of integers. Consider the ramified extension $L = K(\sqrt{t})$. The residue field of $L$ is also $k$. Then $q \otimes L \cong q_1 \perp q_2$. Let $m/k$ be a field extension of degree at most $i_s(k)$ splitting the reduction $\overline{q_1 \perp q_2}$. Let $M/K$ be a lift of the extension $m/k$. By Hensel's lemma, $(q_1 \perp q_2) \otimes L$ is split by $M/L$. Therefore, $i_s(K)$ is at most $2i_s(k)$. 
\end{proof}

\begin{lemma}
	\label{EASYLEMMA}
Let $\{ \phi_1, \phi_2, \cdots \phi_n  \}$ be a set of anisotropic Pfister forms over a field $F$ with $\ch(F) \neq 2$. Let $\phi_i^{\prime}$ denote their pure subforms. The forms $\phi_i$ share a common quadratic splitting field if and only if $\cap_{i=1}^{n}D_F(\phi_i^{\prime}) \neq \varnothing$.
\end{lemma}
\begin{proof}
Suppose all the Pfister forms are split by $F(\sqrt{c})$. In that case $\langle 1, -c \rangle$ is a subform of all the $\phi_i$. In which case $-c \in \cap_{i=1}^{n}D_F(\phi_i^{\prime})$. Conversely if $-c \in \cap_{i=1}^{n}D_F(\phi_i^{\prime})$, all $\phi_i$ have $\langle 1, -c \rangle$ as a subform. Thus all $\phi_{i}$ are split by $F(\sqrt{c})$.  
\end{proof}

\begin{cor}
\label{SPLITCOR} 
Let $\phi$ and $\gamma$ be $m$-fold and $n$-fold anisotropic Pfister forms respectively over a field $F$ with $\ch(F) \neq 2$.  Let $\phi^{\prime}$ and $\gamma^{\prime}$ denote their pure subforms. Then $\phi$ and $\gamma$ share a common quadratic splitting field if and only if $\gamma^{\prime} \perp \langle -1 \rangle \phi^{\prime}$ is isotropic.
\end{cor}
\begin{proof}
Note that $\gamma^{\prime} \perp \langle -1 \rangle \phi^{\prime}$ is anisotropic if and only if $D(\gamma^{\prime}) \cap D(\phi^{\prime}) = \varnothing$. Lemma \ref{EASYLEMMA} then shows the rest.
\end{proof}

\begin{prop}
\label{FUB}
Let $F$ be a field with $\ch(F) \neq 2$. Suppose that for every finite extension $L/F$, and every $\alpha \in H^2(L, \mathbb{Z}/2\mathbb{Z})$, one has $\ind(\alpha)| [\per(\alpha)]^2$. Further the $u$-invariant $u(L) \leq 8$ for every finite extension $L/F$. Then the splitting dimension $F$ is at most $8$.
\end{prop}
\begin{proof}
We will show that any form over any finite extension of $F$ is split by an extension of degree at most $8$. We may as well assume that the form is defined over $F$. Let $q/F$ be that form. We may assume it is even dimensional by replacing it by a codimension one subform if necessary. Let $L/F$ be a quadratic extension which splits its discriminant. Thus in the Witt ring $W(L)$, $q \otimes L$ lies in $I^2L$. Let $\alpha := e_2(q_L)$ be its $e_2$ invariant. By our assumption on the index of $\alpha$, it follows from Albert's theorem that $\alpha = \alpha_1 + \alpha_2$, where $\alpha_i$ are classes of quaternion algebras. We will abuse notation and not distinguish the norm forms of quaternion algebras (which are Pfister forms) and the algebras themselves. Let $M/L$ be a quadratic extension splitting $\alpha_1$. Then $q \otimes M - \alpha_2 \otimes M$ lies in $I^3M$. Thus $q \otimes M = \alpha_2 \otimes M + \beta$, where $\beta$ lies in $I^3M$. Because $u(M) \leq 8$, $\beta$ is similar to a three-fold Pfister form $\tilde{\beta}$ by the Arason Pfister Hauptsatz (see, for example \cite[Chapter X, Theorem 5.6]{Lm05}). But since $u(M) \leq 8$, the form given by the difference of the pure subforms $\alpha_2^{\prime} \otimes M \perp \langle -1 \rangle \tilde{\beta}^{\prime}$ being $10$-dimensional, is isotropic. By Corollary \ref{SPLITCOR}, there is a quadratic extension $K/M$ splitting $\alpha_2 \otimes M$ and $\beta$ simultaneously. Thus $M/F$ splits $q$ and its degree is $8$.  
\end{proof}

\begin{remark}
\begin{enumerate}
\item In view of results of Saltman (see \cite{S97}) on period-index bounds for $p$-adic curves and of Parimala and Suresh (see \cite{PS07}) on the $u$-invariant, the splitting dimension of function fields of curves over $p$-adic fields (for $p \neq 2$) is at most $8$. 

\item The splitting dimension of function fields of surfaces over finite fields (for example $\mathbb{F}_p(x, y))$ and that of fraction fields of complete two dimensional regular local rings with finite residue field (for example $\mathbb{F}_p((x, y))$) is also $8$. Note that such fields are $C_3$ fields. Thus the $u$-invariant of these fields is at most $8$. The fact that index divides the square of the period for surfaces over finite fields was proved by Lieblich (see \cite{L15}). 

\end{enumerate}
\end{remark}

\begin{prop}
\label{dDIM}
Let $F$ be a finitely generated transcendence degree $d$ field over an algebraically closed field $K$ of characteristic unequal to two. Then $i_s(F) \geq 2^{d}$. 
\end{prop}
\begin{proof}
Again, without loss of generality, we may work with $F$ in place of any finite extension of $F$, and show that $F$ admits a quadratic form that needs an extension of degree at least $2^{d}$ to be split.

We first show that there exists a Brauer class $\alpha$ over $F$ of index $2^{d-1}$ and a quadratic extension $L/F$ such that $\ind(\alpha \otimes L) = 2^{d-1}$. Let $p$ be the norm form of $L/F$ and $p_{\alpha}$ be a form with Clifford invariant $\alpha$. Then $q = p \perp p_{\alpha}$ has spitting index at least $2^{d}$.

We show the existence of a class $\alpha$ and field extension $L/F$ by induction on the transcendence degree: The base case $d=1$, follows from the fact that all Brauer classes on function fields of curves over algebraically closed fields are trivial. 

Let $M$ be a transcendence degree $k$ field over $K$. There exists a discrete valuation $v$ on $M$ which is trivial on $K$. Let $\pi_{v} \in M$ be a parameter for this valuation, and $m(v)$ be the residue field. By induction hypothesis, there exists $\alpha_0$ in ${}_{2}\br(m(v))$ of index $2^{k-2}$ and an extension $L := m(v)(\sqrt{u})$ such that $\ind(\alpha_0 \otimes L) = 2^{k-2}$. Consider the class $\alpha = \alpha_0 + (u, \pi_{v})$ in $\br(M_{v})$. One may assume that this class is defined over $M$. Viewing it as a class over $M$, observe that $\ind(\alpha) \leq 2^{k-1}$. Note that by \cite{GS06}[Chapter 6, Exercise 6], $\ind(\alpha \otimes M_{v}) = 2 \ind(\alpha_0 \otimes M_{v}(\sqrt{u})) = 2\ind(\alpha_0 \otimes L) = 2^{k-1}$. This shows that $\ind(\alpha)$ must be $2^{k-1}$. Further, note also that the extension $M(\sqrt{\pi_{v} +1})/M$ splits in $M_{v}$. This is because $\pi_{v} + 1$ is a square in $M_{v}$ by Hensel's Lemma. Thus $\ind(\alpha \otimes M(\sqrt{\pi_{v} + 1})) = 2^{k-1}$. 
\end{proof}
\begin{remark}
As an explicit example, the quadratic form $q =  \langle x, y, xy, -(y+1), -z, z(y+1)\rangle \perp \langle 1, -(z+1) \rangle$ over $\mathbb{C}(x, y, z)$ needs an extension of degree $8$ to be split.
\end{remark}

\comm{
\begin{example}
We give an example of a quadratic form over $\mathbb{C}(x, y, z)$ with splitting index at least $8$.
Consider the division algebra $D = (x, y)\otimes (y+1, z)$. This is a division algebra because the corresponding Albert form $ q = \langle x, y, xy, -(y+1), -z, z(y+1) \rangle$. This form is anisotropic over $\mathbb{C}(x, y)((z))$. Note that $q = \langle x, y, xy, -(y+1) \rangle \perp z \langle 1, -(y+1) \rangle$. Both the forms $\langle x, y, xy, -(y+1) \rangle$ and $\langle 1, -(y+1) \rangle$ are anisotropic over the residue field over the residue field $\mathbb{C}(x, y)$ and therefore by Springer, $q$ is anisotropic over $\mathbb{C}(x, y)((z))$ and also over $\mathbb{C}(x, y, z)$. Note that $z+1$ is a square in $\mathbb{C}(x, y)((z))$ by Hensel's Lemma. Therefore $\mathbb{C}(x, y, z)(\sqrt{z+1})$ is split in $\mathbb{C}(x, y)((z))$. Thus $D$ is a division algebra over $\mathbb{C}(x, y)((z))$ as well as $\mathbb{C}(x, y, z)(\sqrt{z+1})$. Therefore $p = \langle x, y, xy, -(y+1), -z, z(y+1)\rangle \perp \langle 1, -(z+1) \rangle$ is our required quadratic form.  
\end{example}

}
 
\begin{prop}
	\label{CDV}
Let $F$ be the function field of a curve over a complete discretely valued field with residue characteristic unequal to two. Suppose that over every finite extension $L/F$, there exists a class $\alpha \in {}_{2}\br(L)$ such that one has $\ind(\alpha) = 2^{n}$. Then $i_s(F) \geq 2^{n+1}$.
\end{prop}
\begin{proof}
Again, without loss of generality, we work with $F$ in place of a finite extension of $F$. Let $\alpha$ be in ${}_{2}\br(F)$ such that $\ind(\alpha) = 2^n$. By the Merkurjev Suslin Theorem, there is a quadratic form $p_{\alpha}$ of trivial discriminant such that $e_2(p_{\alpha}) = \alpha$. By \cite[Theorem 2.6]{RS13}, there exists a non-trivial discrete valuation $v$ of $F$ such that $\ind(\alpha \otimes F_{v}) = 2^n$. Let $L/F$ be  a quadratic extension which is split over $v$. Let $p$ be the norm form of $L/F$. Then the form $q = p \perp p_{\alpha}$ needs an extension of degree at least $2^{n}$ to be split. Any extension $M/F$ splitting $q$ must contain $L/F$. Since $L$ is split at $v$, one has $\ind(\alpha \otimes L_v) = 2^{n}$, and hence $\ind(\alpha) = 2^{n}$. As a result $[M : L]$ must be divisible by $2^{n}$. This shows that $i_s(F) \geq 2^{n + 1}$.
\end{proof}

\section{Acknowledgements}
This work will be a part of my doctoral thesis. I would like to express my heartfelt gratitude to my advisor Prof. Daniel Krashen for suggesting this problem to me and also thank him for his time, patience and constant encouragement along the way. His feedback on the writing has been very valuable and instructive. I also thank Prof.~David Harbater for numerous comments and suggestions.


\begin{thebibliography}{G19}

\bibitem[AimPL]{AimPL11} 
\newblock{\em Deformation theory and the Brauer group}.
\newblock Available at http://aimpl.org/deformationbrauer  2011.

\bibitem[AAIKL]{AAIKL17} 
Asher Auel, Benjamin Antieau, Colin Ingalls, Daniel Krashen, Max Lieblich.
\newblock{\em Period-index bounds for arithmetic threefolds}. 
\newblock 2017 manuscript. To appear in Inventiones Mathematicae.

\bibitem[CT]{CT92} 
Jean-Louis Colliot-Th\'el\`ene. 
\newblock {\em Birational invariants, purity and the Gersten conjecture. K-theory and algebraic geometry: connections with quadratic forms and division algebras} (Santa Barbara, CA, 1992) . 
\newblock 1--64, Proc. Sympos. Pure Math., 58, Part 1, Amer. Math. Soc., Providence, RI, 1995.

\bibitem[GLL]{GLL13}
Ofer Gabber, Qing Liu, Dino Lorenzini.
{\em The index of an algebraic variety}.  
\newblock Invent. Math. 192 (2013), no. 3, 567-626.

\bibitem[Gr]{GR71}
Alexander Grothendieck.
{ \em Rev\^{e}tements \'{e}tales et groupe fondamental ({SGA} 1)}. 
\newblock S{\'e}minaire de g{\'e}om{\'e}trie alg{\'e}brique du Bois Marie
  1960--61.
\newblock Lecture Notes in Mathematics, vol.~224, Springer-Verlag, 1971. 

\bibitem[GS]{GS06}
Philippe Gille, Tam\'{a}s Szamuely.
{\em  Central simple algebras and Galois cohomology}.
\newblock Cambridge Studies in Advanced Mathematics, 101. Cambridge University Press, Cambridge, 2006. xii+343 pp. ISBN: 978-0-521-86103-8; 0-521-86103-9.  

\bibitem[H]{H13}
David Harbater
 {\em Patching in Algebra: Notes from the Luxembourg Winter School}. 
 \newblock Available at https://www.math.upenn.edu/~harbater/luxtalks.pdf. , 2012

\bibitem[HH]{HH10}
David Harbater, Julia Hartmann.
 {\em Patching over fields}. 
 \newblock Israel J. Math. 176 (2010), 61-107.

\bibitem[HHK09]{HHK09}
David Harbater, Julia Hartmann, Daniel Krashen.
 {\em Applications of patching to quadratic forms and central simple algebras}.  
 \newblock Invent. Math. 178 (2009), no. 2, 231-263. 
 
 \bibitem[HHK14]{HHK14}
 David Harbater, Julia Hartmann, Daniel Krashen. 
 {\em Local-global principles for Galois cohomology}.  
 \newblock Comment. Math. Helv. 89 (2014), no. 1, 215-253.
 
 \bibitem[HHK15]{HHK15}
 David Harbater, Julia Hartmann, Daniel Krashen. 
 {\em Local-Global principles for torsors over arithmetic curves}. 
 \newblock Amer. J. Math. 137 (2015), no. 6, 1559-1612.

\bibitem[HHK15(1)]{HHK15(1)}
David Harbater, Julia Hartmann, Daniel Krashen. 
{\em Refinements to patching and applications to field invariants}.  
\newblock Int. Math. Res. Not. IMRN 2015, no. 20, 10399-10450.


\bibitem[HHKPS]{HHKPS17}
David Harbater, Julia Hartmann, Daniel Krashen, R. Parimala, V. Suresh. 
{\em Local-global Galois theory of arithmetic function fields}.
\newblock Israel J. Math. 232 (2019), no. 2, 849–882.

\bibitem[K]{K16}
Daniel Krashen.
{\em Period and index, symbol lengths, and generic splittings in Galois cohomology}.   
\newblock Bull. Lond. Math. Soc. 48 (2016), no. 6, 985-1000.

\bibitem[L]{L15}
Max Lieblich.
{\em The period-index problem for fields of transcendence degree $2$}.  
\newblock Ann. of Math. (2) 182 (2015), no. 2, 391-427.

\bibitem[Lm]{Lm05}
T. Y. Lam.
{\em Introduction to quadratic forms over fields}. 
\newblock Graduate Studies in Mathematics, 67. American Mathematical Society, Providence, RI, 2005. xxii+550 pp. ISBN: 0-8218-1095-2.

\bibitem[Li]{Li02}
Qing Liu.
{\em Algebraic geometry and arithmetic curves}. 
\newblock Translated from the French by Reinie Erne. 
\newblock Oxford Graduate Texts in Mathematics, 6. Oxford Science Publications. Oxford University Press, Oxford, 2002. xvi+576 pp. ISBN: 0-19-850284-2.

\bibitem[Mi]{Mi80}
James S. Milne.
{\em \'{E}tale cohomology}. 
\newblock Princeton Mathematical Series, 33. Princeton University Press, Princeton, N.J., 1980. xiii+323 pp. ISBN: 0-691-08238-3. 

\comm{
\bibitem[Pa]{Pa03}
I. A. Panin.
{\em The equicharacteristic case of the Gersten conjecture}.  
\newblock Tr. Mat. Inst. Steklova 241 (2003).
}

\bibitem[PS]{PS07}
R. Parimala, V. Suresh.
{\em The $u$-invariant of the function fields of $p$-adic curves}. 
\newblock Ann. of Math. (2) 172 (2010), no. 2, 1391-1405.

\bibitem[PS15]{PS15}
R. Parimala, V. Suresh.
{\em On the u-invariant of function fields of curves over complete discretely valued fields}.
\newblock  Adv. Math. 280 (2015), 729–742.


\bibitem[RS]{RS13}
B. Surendranath Reddy, V. Suresh.
{\em Admissibility of groups over function fields of $p$-adic curves}. 
\newblock Adv. Math. 237 (2013), 316-330.

\bibitem[SM]{SM19}
Makoto Sakagaito
{\em A note on Gersten's conjecture for \'etale cohomology over two dimensional henselian regular local rings}
\newblock Manuscript available at https://arxiv.org/abs/1903.01677v1 , 2019

\bibitem[S]{S97}
David J. Saltman.
{\em Division algebras over $p$-adic curves}.  
\newblock J. Ramanujan Math. Soc. 12 (1997), no. 1, 25-47.

\bibitem[S2]{S98}
David J. Saltman.
{\em Correction to: "Division algebras over p-adic curves}.
\newblock J. Ramanujan Math. Soc. 13 (1998), no. 2, 125–129.
\end{thebibliography}
\end{document}